\documentclass[11pt,reqno]{amsart}
\usepackage{graphicx}

\usepackage{amssymb}
\usepackage{cite}
\usepackage{amsmath}
\usepackage{latexsym}
\usepackage{amscd}
\usepackage{amsthm}
\usepackage{mathrsfs}
\usepackage{url}

\usepackage[backref=page,linktocpage=true,colorlinks,citecolor=magenta,linkcolor=blue,urlcolor=magenta]{hyperref}
\numberwithin{equation}{section}

\newtheorem{thm}{Theorem}
\newtheorem{corr}[thm]{Corollary}
\newtheorem{lem}[thm]{Lemma}
\newtheorem{prop}[thm]{Proposition}

\theoremstyle{definition}
\newtheorem{defn}{Definition}[section]
\theoremstyle{remark}
\newtheorem{rem}[thm]{Remark}

\def\R{\mathbb R}

\def\H{\mathbb H}
\def\SS{\mathbb S}

\def\PW{\partial\mathcal{W}}

\def\f{\frac}

\def\pt{\partial}

\begin{document}
\title[Sharp upper bounds for the capacity]{Sharp upper bounds for the capacity in the hyperbolic and Euclidean spaces}
\author[H.~Li]{Haizhong~Li}
\address{Department of Mathematical Sciences, Tsinghua University, Beijing 100084, P.~R.~China}
\email{\href{mailto:lihz@tsinghua.edu.cn}{lihz@tsinghua.edu.cn}}
\author[R.~Li]{Ruixuan~Li}
\address{Department of Mathematical Sciences, Tsinghua University, Beijing 100084, P.~R.~China}
\email{\href{mailto:li-rx18@tsinghua.org.cn}{li-rx18@tsinghua.org.cn}}
\author[C.~Xiong]{Changwei~Xiong}
\address{School of Mathematics, Sichuan University, Chengdu 610065, Sichuan,  P.~R.~China}
\email{\href{mailto:changwei.xiong@scu.edu.cn}{changwei.xiong@scu.edu.cn}}
\date{\today}
\thanks{The first and second authors were partially supported by NSFC Grant no.~11831005. The third author was supported by National Key R and D Program of China 2021YFA1001800, NSFC Grant no.~12171334, and the funding (no.~1082204112549) from Sichuan University.}

\maketitle

\begin{abstract}
We derive various sharp upper bounds for the $p$-capacity of a smooth compact set $K$ in the hyperbolic space $\mathbb{H}^n$ and the Euclidean space $\mathbb{R}^n$. Firstly, using the inverse mean curvature flow, for the mean convex and star-shaped set $K$ in $\mathbb{H}^n$, we obtain sharp upper bounds for the $p$-capacity $\mathrm{Cap}_p(K)$ in three cases: (1) $n\geq 2$ and $p=2$, (2) $n=2$ and $p\geq 3$, (3) $n=3$ and $1<p\leq 3$; Using the unit-speed normal flow, we prove a sharp upper bound for $\mathrm{Cap}_p(K)$ of a convex set $K$ in $\H^n$ for $n\geq 2$ and $p>1$. Secondly, for the compact set $K$ in $\mathbb{R}^3$, using the weak inverse mean curvature flow, we get a sharp upper bound for the $p$-capacity ($1<p<3$) of the set $K$ with connected boundary; Using the inverse anisotropic mean curvature flow, we deduce a sharp upper bound for the anisotropic $p$-capacity ($1<p<3$) of an $F$-mean convex and star-shaped set $K$ in $\R^3$.
\end{abstract}

\section{Introduction}\label{sec1}

The capacity of compact sets in a Riemannian manifold is an important geometric quantity and admits many crucial applications in the topics such as parabolicity criteria, eigenvalue estimates, heat kernel estimates, etc. See, e.g., \cite{Gri99,Gri99a} for an overview of the applications. 

For the problem to estimate the $p$-capacity in the Euclidean space, the asymptotically flat Riemannian manifold and some other Riemannian manifolds (mainly modelled by the Euclidean space), there have been extensive investigations; see, e.g., the works \cite{AFM20,AM15,AM20,LX22,PS51,Xiao17,XY22,Xiao16,FS14,BM08,Maz11,HKM06,LXZ11,Xiao17a}. In contrast, we find there are few results in the hyperbolic space, partially because even for the model set, the geodesic ball $B_r\subset \H^n$, the $p$-capacity $\mathrm{Cap}_p(\overline{B_r})$ does not have a simple expression. See \cite{Jin22} for some limit results on the $p$-capacity of geodesic balls in an asymptotically hyperbolic manifold. To fill this gap in the literature, our first main focus in this paper is to derive various sharp upper bounds for the $p$-capacity of compact sets in the hyperbolic space.

Let $K\subset \H^n$ be a compact set in the hyperbolic space. For $p\geq 1$, define the $p$-capacity of $K$ as
\begin{equation*}
\mathrm{Cap}_p(K)=\inf \int_{\H^n} |\nabla f|^pdv,
\end{equation*}
where the infimum is taken over all smooth functions $f$ with compact support satisfying $f=1$ on $K$. For the $p$-capacity in a general Riemannian manifold, see, e.g., \cite{Gri99,Gri99a}.

Let $p>1$. Assume $U:=\H^n\setminus K$ is foliated by a family of hypersurfaces $M_t$ with $M_0=\pt K$. It has been derived in, e.g., \cite{Gri99} that
\begin{align}\label{eq-general-bound}
\mathrm{Cap}_p(K)\leq \left(\int_0^\infty T_p^{-1/(p-1)}(t)dt\right)^{1-p},
\end{align}
where the function $T_p(t)$ is defined as
\begin{equation*}
T_p(t)=\int_{M_t}|D\psi|^{p-1}d\mu_t,
\end{equation*}
and $\psi:\H^n\setminus K\rightarrow \R$ is the level-set function determined by the hypersurfaces $M_t$
\begin{equation*}
\psi(x)=t,\quad \text{when }x\in M_t.
\end{equation*}

In the first part of the paper, we will utilize the estimate \eqref{eq-general-bound} to get four kinds of upper bounds for the $p$-capacity of smooth compact sets $K$ in $\H^n$ in terms of certain geometric quantities of $K$ (see Theorems~\ref{thm1}--\ref{thm4}).

\begin{thm}\label{thm1}
Let $K\subset \H^n$ ($n\geq 2$) be a compact set with smooth, mean convex and star-shaped boundary. Then
\begin{align*}
&\qquad \qquad \mathrm{Cap}_2(K)\leq n(n-1)\times  \\
& \left(\int_0^\infty \left(e^{\frac{n-2}{n-1}t}\left(W_2(K)+\frac{n-2}{n(n-1)}\int_0^t e^{-\frac{n-2}{n-1}\tau}I(e^\tau |\pt K|)d\tau\right)+\frac{1}{n}I(e^t |\pt K|)\right)^{-1}dt\right)^{-1},
\end{align*}
where the second quermassintegral $W_2(K)$ is given by
\begin{align*}
W_2(K)&:=\frac{1}{n(n-1)}\int_{\pt K}\sigma_1d\mu-\frac{1}{n}|K|,
\end{align*}
with $\sigma_1$ being the mean curvature of the boundary $\pt K$ (the summation of principal curvatures of $\pt K$), and $I:\R^+\to \R^+$ is the isoperimetric function in $\H^n$, i.e., $I$ is defined by $|B_r|=I(|\pt B_r|)$ ($r>0$). In particular, when $n=2$, we get
\begin{align*}
\mathrm{Cap}_2(K)&\leq \left(\int_0^\infty (4\pi^2+e^{2t}|\pt K|^2)^{-1/2}dt\right)^{-1}=\frac{2\pi}{\mathrm{arsinh}( 2\pi/|\pt K|)}.
\end{align*}
Moreover, the equalities hold if and only if $K$ is a geodesic ball.
\end{thm}
\begin{thm}\label{thm2}
Let $K\subset \H^2$ be a compact set with smooth convex boundary $M=\pt K$ and $p\geq 3$. Then
\begin{align*}
\mathrm{Cap}_p(K)\leq \left(\int_0^\infty e^{\frac{p-2}{p-1}t}|M|^{\frac{p-2}{p-1}} \left(e^{2t}|M|^2+ \left(|M|^{p-2}T_p(0)\right)^{\frac{2}{p-1}}-|M|^2\right)^{-\frac{1}{2}}dt\right)^{1-p},
\end{align*}
where $T_p(0)=\int_M \sigma_1^{p-1}d\mu$. Moreover, the equality holds if and only if $K$ is a geodesic ball.
\end{thm}
\begin{thm}\label{thm3}
Let $K\subset \H^3$ be a compact set with smooth, mean convex and star-shaped boundary $M=\pt K$ and $1<p\leq 3$. Then
\begin{align*}
&\mathrm{Cap}_p(K)\leq \\
& \left(\int_0^\infty\left(\left(\int_M \sigma_1^2d\mu-4|M|-16\pi\right)e^{-t}+4|M|e^{t}+16\pi\right)^{-\frac{1}{2}}e^{-\frac{3-p}{2(p-1)}t}dt\right)^{1-p}|M|^{\frac{3-p}{2}}.
\end{align*}
Moreover, the equality holds if and only if $K$ is a geodesic ball.
\end{thm}
\begin{thm}\label{thm4}
Let $K\subset \H^n$ ($n\geq 2$) be a compact set with smooth convex boundary $M=\pt K$ and $p>1$. Then
\begin{align*}
&\mathrm{Cap}_p(K)\leq \left(\int_0^\infty\left( \int_M \sum_{i=0}^{n-1}\cosh^{n-1-i}t\cdot \sinh^i t\cdot \sigma_id\mu\right)^{-\frac{1}{p-1}} dt\right)^{1-p} ,
\end{align*}
where $\sigma_i$ denotes the $i$th mean curvature of the boundary $M$. Moreover, the equality holds if $K$ is a geodesic ball.
\end{thm}
For the proofs of Theorems~\ref{thm1}, \ref{thm2}, and \ref{thm3}, we use the inverse mean curvature flow (see Gerhardt \cite{Ger11}) to generate the foliation $M_t$; while for the proof of Theorem~\ref{thm4}, we use the unit-speed normal flow to get the foliation $M_t$.

Next, we come to the second part of the paper, the cases in the Euclidean space $\R^n$. Different from in the hyperbolic space, in $\R^n$ we may consider the more general anisotropic $p$-capacity for a compact set. That is, given a Minkowski norm $F$ on $\R^n$, we define the anisotropic $p$-capacity ($1<p<n$) of a compact set $K\subset \R^n$ by
\begin{align*}
\mathrm{Cap}_{F,p}(K) = \inf \left\{ \int_{\R^n} F^p(Dv) dx: v \in C_c^{\infty}(\R^n), v \geq 1\text{ on }K \right\},
\end{align*}
where $C_c^\infty(\R^n)$ is the set of smooth functions with compact support in $\R^n$. See Section~\ref{sec2.3} for more details on the definitions. In particular, when $F$ is the Euclidean norm ($F(\xi)=|\xi|$), the anisotropic $p$-capacity reduces to the ordinary $p$-capacity in $\R^n$.

First we get a sharp upper bound for the (ordinary) $p$-capacity of a compact set $K\subset \R^3$ with smooth and connected boundary $\Sigma=\pt K$. In $\R^3$ we use $H$ for $\sigma_1$ to denote the mean curvature of a surface $\Sigma\subset \R^3$.
\begin{thm}\label{thm5}
Let $K\subset \R^3$ be a compact set with smooth connected boundary $\Sigma=\pt K$ and $1<p< 3$. Then we have two cases: (1) $\int_\Sigma H^2 d\mu=16\pi$, $\Sigma$ is a round sphere with radius $r_0>0$ and
\begin{equation*}
\mathrm{Cap}_p(K)=\left(\frac{3-p}{p-1}\right)^{p-1}4\pi r_0^{3-p};
\end{equation*}
(2) $\int_\Sigma H^2 d\mu>16\pi$ and
\begin{align*}
\mathrm{Cap}_{ p}(K) &<  \left( \frac{3-p}{p-1} \right)^{p-1} (4\pi)^{\frac{p-1}{2}} |\Sigma|^{\frac{3-p}{2}} s^{\frac{3-p}{2}}  \theta^{1-p},
\end{align*}
where
\begin{align*}
s&=\frac{\int_{\Sigma}H^2 d\mu}{16\pi}-1 \text{ and }
\theta = \int_{0}^{s^{\frac{3-p}{2(p-1)}}} \left(1 +  r^{\frac{2(p-1)}{3-p}}\right)^{-\frac{1}{2}}\,dr.
\end{align*}
\end{thm}
\begin{rem}
For a closed surface $\Sigma\subset \R^3$, the Willmore energy $\int_\Sigma H^2d\mu$ satisfies $\int_\Sigma H^2d\mu \geq 16\pi$ with the equality if and only if $\Sigma$ is a round sphere. See, e.g., \cite{Che71,RS10,Wil68}.
\end{rem}
\begin{rem}
Theorem~\ref{thm5} with $p=2$ was proved by the third author in \cite{Xio22} in order to get sharp estimates for an exterior Steklov eigenvalue problem. Besides, Theorem~\ref{thm5} improves on partial results in \cite{BM08,Xiao16}.
\end{rem}

Second we obtain a sharp upper bound for the anisotropic $p$-capacity of compact sets in the Euclidean space $\R^3$. For the Wulff shape $\pt\mathcal{W}$ and anisotropic quantities in Theorem~\ref{thm6}, we refer to Section~\ref{sec2.3}.
\begin{thm}\label{thm6}
Let $K\subset \R^3$ be a compact set with smooth, anisotropically mean convex and star-shaped boundary $\Sigma=\pt K$ and $1<p< 3$. Then we have two cases: (1) $\int_{\Sigma} H_F^2\,d\mu_F=4|\PW|_F$, the surface $\Sigma$ is a translated scaled Wulff shape $r_0\pt \mathcal{W}+x_0$ ($r_0>0$, $x_0\in \R^3$) and
\begin{align*}
\mathrm{Cap}_{F, p}\left(K\right) = \left(\frac{3-p}{p-1}\right)^{p-1} |\PW|_F \; r_0^{3-p};
\end{align*}
(2) $\int_{\Sigma} H_F^2\,d\mu_F>4|\PW|_F$ and
\begin{align*}
\mathrm{Cap}_{F, p}(K) &< \left( \frac{3-p}{p-1} \right)^{p-1} |\PW|_F^{\frac{p-1}{2}} |\Sigma|_F^{\frac{3-p}{2}} s^{\frac{3-p}{2}}  \theta^{1-p},
\end{align*}
where
\begin{align*}
s=\frac{\int_{\Sigma} H_F^2\,d\mu_F}{4|\PW|_F}  - 1,\text{ and } \theta = \int_{0}^{s^{\frac{3-p}{2(p-1)}}} \left(1 +  r^{\frac{2(p-1)}{3-p}}\right)^{-\frac{1}{2}}\,dr.
\end{align*}
\end{thm}
When $p=2$, we get the following corollary.
\begin{corr}\label{corr1}
Let $K\subset \R^3$ be a compact set with smooth, anisotropically mean convex and star-shaped boundary $\Sigma=\pt K$. Then we have two cases: (1) $\int_{\Sigma} H_F^2\,d\mu_F=4|\PW|_F$, the surface $\Sigma$ is a translated scaled Wulff shape $r_0\pt \mathcal{W}+x_0$ ($r_0>0$, $x_0\in \R^3$) and
\begin{align*}
\mathrm{Cap}_{F, 2}\left(K\right) =  |\PW|_F \; r_0;
\end{align*}
(2) $\int_{\Sigma} H_F^2\,d\mu_F>4|\PW|_F$ and
\begin{equation}\label{eq-new-bound}
\mathrm{Cap}_{F,2}( K) < \sqrt{|\PW|_F |\Sigma|_F} \frac{\sqrt{s}}{\mathrm{arsinh}\sqrt{s}},
\end{equation}
where
\begin{align*}
s=  \frac{\int_{\Sigma} H_F^2\,d\mu_F}{4|\PW|_F}-1.
\end{align*}
\end{corr}
\begin{rem}
For the isotropic case (i.e., $F$ is the Euclidean norm $F(\xi)=|\xi|$), Theorem~\ref{thm6} is contained in Theorem~\ref{thm5}. In addition, Theorem~\ref{thm6} and Corollary~\ref{corr1} improve on the corresponding results in \cite{LX22}. For example, under the same conditions, the bound (1.7) in \cite{LX22} reads (if $\Sigma$ is not a Wulff shape)
\begin{align}\label{eq-old-bound}
\mathrm{Cap}_{F,2}\left(K\right) < \frac{1}{2} \sqrt{|\mathcal{\pt W}|_F |\Sigma|_F} \left(1 + \sqrt{\frac{1}{4 |\mathcal{\pt W}|_F} \int_{\Sigma} H_F^2\,d\mu_F}\right).
\end{align}
By an elementary computation, we see that the bound \eqref{eq-new-bound} is better than the bound \eqref{eq-old-bound}.
\end{rem}
For the proof of Theorem~\ref{thm5}, we use the weak inverse mean curvature flow due to Huisken and Ilmanen \cite{HI01}; while for that of Theorem~\ref{thm6}, we employ the inverse anisotropic mean curvature flow (see Xia \cite{Xia17}). Under these flows we can foliate the domain $\R^n\setminus K$ and apply the classical approach as in \cite{PS51,Sze31}. This approach has also been applied in other works, and the most relevant papers to Theorems~\ref{thm5} and \ref{thm6} are \cite{BM08,Xiao16,LX22,Xio22} as mentioned before.

The paper is organized as follows. In Section~\ref{sec2} we provide some preliminaries and backgrounds, mainly on the inverse mean curvature flow in the hyperbolic space, the weak inverse mean curvature flow in the Euclidean space, the anisotropic geometry in the Euclidean space and the related inverse anisotropic mean curvature flow. In Section~\ref{sec3} we employ the inverse mean curvature flow in $\H^n$ to prove Theorems~\ref{thm1}, \ref{thm2}, and \ref{thm3}. In the next Section~\ref{sec4} we use the unit-speed normal flow to prove Theorem~\ref{thm4}. In Sections~\ref{sec5} and \ref{sec6} we discuss the problems in $\R^3$. In Section~\ref{sec5} we use the weak inverse mean curvature flow to prove Theorem~\ref{thm5}. In Section~\ref{sec6} we apply the inverse anisotropic mean curvature flow to prove Theorem~\ref{thm6}.

\section{Preliminaries}\label{sec2}

In this section we recall some backgrounds and tools which will be used later in the paper. Throughout the paper, $|\cdot|$ denotes the volume of compact sets with non-empty interior or the area of hypersurfaces, $\sigma_i$ ($1\leq i\leq n-1$) denotes the $i$th mean curvature of a hypersurface $M^{n-1}$ in an $n$-dimensional Riemannian manifold (occasionally, we will use $H$ for $\sigma_1$), and other standard notations will also be used without further indication whenever no confusion arises.

\subsection{The quermassintegrals and the inverse mean curvature flow in $\mathbb{H}^n$}

For a compact set $K\subset \H^n$ in the hyperbolic space, the quermassintegrals are important geometric quantities. In this paper we shall use the first three of them as follows:
\begin{align*}
W_0(K)&:=|K|,\\
W_1(K)&:=\frac{1}{n}|\pt K|,\\
W_2(K)&:=\frac{1}{n(n-1)}\int_{\pt K}\sigma_1d\mu-\frac{1}{n}|K|.
\end{align*}
See, e.g., \cite{WX14} for more details on the quermassintegrals.

Next we briefly introduce the inverse mean curvature flow in the hyperbolic space due to C.~Gerhardt \cite{Ger11}. Consider the inverse mean curvature flow for a mean convex hypersurface in $\H^n$ ($n\geq 2$), i.e., the parabolic evolution equation
\begin{align}\label{eq-IMCF-h}
\partial_t X = \frac{1}{\sigma_1} \nu,
\end{align}
where $X:N\times [0,T)\to \H^n$ is a smooth family of embeddings from a closed manifold $N$ to $\H^n$ and $\nu$ denotes the outward unit normal vector on the flow hypersurface.

Given an initial mean convex and star-shaped hypersurface $M_0=X(N,0)$, Gerhardt \cite{Ger11} proved the long-time existence and the smooth and exponential convergence of the flow \eqref{eq-IMCF-h} as follows.
\begin{thm}[Gerhardt \cite{Ger11}]
The flow \eqref{eq-IMCF-h} with an initial mean convex and star-shaped hypersurface $M_0=X(N,0)$ exists for all time $t\in [0,\infty)$. The flow hypersurfaces $M_t$ converge to infinity, and become strictly convex exponentially fast and also more and more totally umbilic. In fact there holds
\begin{equation*}
|h^i_j-\delta^i_j|\leq ce^{-t/(n-1)},
\end{equation*}
i.e., the principal curvatures are uniformly bounded and converge exponentially fast to $1$.
\end{thm}
In addition, we need the evolution equations for the area element $d\mu_t$ and the mean curvature $\sigma_1$, which is standard and may be found in, e.g., Proposition~3 in \cite{LWX14}.
\begin{prop}
Along the flow \eqref{eq-IMCF-h}, there holds
\begin{align*}
&\frac{\pt }{\pt t}d\mu_t=d\mu_t,\\
&\pt_t \sigma_1=-\Delta \frac{1}{\sigma_1}-\frac{1}{\sigma_1}|h|^2+\frac{n-1}{\sigma_1},
\end{align*}
where $\Delta$ is the Laplacian on the flow hypersurface $M_t$ and $|h|^2$ denotes the squared norm of the second fundamental form of $M_t$. In particular, $|M_t|=e^t|M|$.
\end{prop}

\subsection{The weak inverse mean curvature flow in $\mathbb{R}^3$}

In this subsection, we introduce the weak inverse mean curvature flow for the compact set $K\subset \R^3$ with smooth and connected boundary $\Sigma=\pt K$, which is due to Huisken and Ilmanen \cite{HI01}. By \cite{HI01}, there exists a proper, Lipschitz function $\phi\geq 0$ on $\overline{U}$ (here $U:=\R^3\setminus K$), called the solution to the weak inverse mean curvature flow with the initial surface $\Sigma$, satisfying the following properties:
\begin{enumerate}
  \item The function $\phi$ takes value $\phi|_\Sigma=0$ and $\lim_{x\to \infty} \phi(x)=\infty$. For $t>0$, the sets $\Sigma_t=\pt\{\phi\geq t\}$ and $\Sigma_t'=\pt \{\phi >t\}$ define two increasing families of $C^{1,\alpha}$ surfaces.
  \item For $t>0$, the surfaces $\Sigma_t$ ($\Sigma_t'$, resp.) minimize (strictly minimize, resp.) area among surfaces homologous to $\Sigma_t$ in the region $\{\phi\geq t\}$. The surface $\Sigma'=\pt \{\phi>0\}$ strictly minimizes area among surfaces homologous to $\Sigma$ in $U$.
  \item For almost all $t>0$, the weak mean curvature $H$ of $\Sigma_t$ is well defined and equals $|\nabla \phi|$, which is positive for almost all $x\in \Sigma_t$.
  \item For each $t>0$, the area $|\Sigma_t|=e^t|\Sigma'|$; and $|\Sigma_t|=e^t|\Sigma|$ if $\Sigma$ is outer-minimizing (i.e., $\Sigma$ minimizes area among all surfaces homologous to $\Sigma$ in $U$).
  \item All the surfaces $\Sigma_t$ ($t>0$) remain connected. The Hawking mass
  \begin{equation}
  m_H(\Sigma_t)=\sqrt{\frac{|\Sigma_t|}{16\pi}}\left(1-\frac{1}{16\pi}\int_{\Sigma_t}H^2d\mu_t\right)
  \end{equation}
  satisfies $\lim_{t\to 0+}m_H(\Sigma_t)\geq m_H(\Sigma')$ and its right-lower derivative satisfies (see (5.24) in \cite{HI01}; but note the misprints on some coefficients there, and the correct coefficients are as in (5.22) in \cite{HI01})
  \begin{align*}
  &\underline{D}_+m_H(\Sigma_t):=\liminf_{s\to t+}\frac{m_H(\Sigma_s)-m_H(\Sigma_t)}{s-t}\\
  &\geq \sqrt{\frac{|\Sigma_t|}{16\pi}}\frac{1}{16\pi}\left(8\pi -4\pi \chi(\Sigma_t)+\int_{\Sigma_t}\left(2|\nabla \log H|^2+\frac{1}{2}(\lambda_1-\lambda_2)^2\right)d\mu_t\right),
  \end{align*}
  where $\lambda_1$ and $\lambda_2$ are the weak principal curvatures of $\Sigma_t$ and $\chi(\Sigma_t)$ is the Euler characteristic of $\Sigma_t$. See \cite[Section~5]{HI01} for more details.
\end{enumerate}

Furthermore, in \cite{Xio22}, the third author defines a modified Hawking mass
\begin{equation}
\widetilde{m}_H(\Sigma_t):=\sqrt{\frac{|\Sigma_t|}{16\pi}}m_H(\Sigma_t)=\frac{|\Sigma_t|}{16\pi}\left(1-\frac{1}{16\pi}\int_{\Sigma_t}H^2d\mu_t\right),
\end{equation}
and finds $\lim_{t\to 0+}\widetilde{m}_H(\Sigma_t)\geq \widetilde{m}_H(\Sigma')$ and
\begin{align*}
&\underline{D}_+\widetilde{m}_H(\Sigma_t)
=\sqrt{\frac{|\Sigma_t|}{16\pi}}\frac{1}{2}m_H(\Sigma_t)+\sqrt{\frac{|\Sigma_t|}{16\pi}}\underline{D}_+m_H(\Sigma_t)\\
&\geq \frac{|\Sigma_t|}{(16\pi)^2}\left(\left(8\pi-\frac{1}{2}\int_{\Sigma_t}H^2d\mu_t\right)\right.\\
&\left.\quad +8\pi -4\pi \chi(\Sigma_t)+\int_{\Sigma_t}\left(2|\nabla \log H|^2+\frac{1}{2}(\lambda_1-\lambda_2)^2\right)d\mu_t\right)\\
&= \frac{|\Sigma_t|}{(16\pi)^2}\left(16\pi -4\pi \chi(\Sigma_t)+\int_{\Sigma_t}(2|\nabla \log H|^2-2\lambda_1 \lambda_2)d\mu_t\right)\\
&= \frac{|\Sigma_t|}{(16\pi)^2}\left(16\pi -8\pi \chi(\Sigma_t)+\int_{\Sigma_t}2|\nabla \log H|^2d\mu_t\right)\\
&\geq 0,
\end{align*}
where the last equality holds because of the weak Gauss--Bonnet formula (Page~403 in \cite{HI01}) and the last inequality is due to the fact that the surfaces $\Sigma_t$ remain connected.
\begin{rem}
The introduction of the modified Hawking mass is inspired by the work \cite{HW15}.
\end{rem}

\subsection{The anisotropic $p$-capacity and the inverse anisotropic mean curvature flow in $\mathbb{R}^3$}\label{sec2.3}
For the anisotropic $p$-capacity nice references include \cite[Section~2.2]{Maz11} and \cite[Chapters~2 and 5]{HKM06}. First we introduce the Minkowski norm on $\R^n$.
\begin{defn}
A function $F \in C^{\infty}(\mathbb{R}^{n} \setminus \{0\})\cap C(\R^n)$ is called a \emph{Minkowski norm} if
\begin{enumerate}
  \item $F$ is a convex, even, $1$-homogeneous function, and $F(\xi) > 0$ if $\xi\neq 0$;
  \item $F$ satisfies the uniformly elliptic condition, i.e., $\mathrm{Hess}_{\R^{n}}(F^2)$ is positive definite in $\R^{n}\setminus \{0\}$.
\end{enumerate}
\end{defn}

Let $K\subset \R^n$ be a compact set. For $n\geq 2$ and $1<p<n$, the \emph{anisotropic $p$-capacity} of $K$ is defined as
\begin{align*}
\mathrm{Cap}_{F,p}(K) = \inf \left\{ \int_{\R^n} F^p(Dv) dx: v \in C_c^{\infty}(\R^n), v \geq 1\text{ on }K \right\},
\end{align*}
where $C_c^\infty(\R^n)$ is the set of smooth functions with compact support in $\R^n$.

Next we review the anisotropic geometry of hypersurfaces in the Euclidean space which is classical in the differential geometry. 

Let $F$ be a Minkowski norm on $\R^n$. We can define its dual norm $F^0$ as follows.
\begin{defn}
The \emph{dual norm} $F^0$ of $F$ is defined as
\begin{equation*}
F^0(x) = \sup_{\xi\neq 0} \f{\langle \xi, x \rangle}{F(\xi)}.
\end{equation*}
\end{defn}
It is known that $F^0$ is also a Minkowski norm.

Recall that $F$ and $F^0$ satisfy the following properties, which are very useful when we want to understand the relationship between the unit normal $\nu$ and the anisotropic unit normal $\nu_F$ of a hypersurface below.
\begin{prop}
\begin{enumerate}
  \item $F(DF^0(x)) = 1$, $F^0(DF(\xi)) = 1$.
  \item $F^0(x) DF(DF^0(x)) = x$, $F(\xi) DF^0(DF(\xi)) = \xi$.
\end{enumerate}
\end{prop}


Next we define the Wulff ball and the Wulff shape determined by $F$.
\begin{defn}
The \emph{Wulff ball} $\mathcal{W}$ centered at the origin is defined as
\begin{equation*}
\mathcal{W} := \{ x\in\mathbb{R}^{n}: F^0(x) < 1 \}.
\end{equation*}
Its boundary $\partial \mathcal{W}$ is called the \emph{Wulff shape}.
\end{defn}


Given a Wulff ball $\mathcal{W}$, we can recover $F$ as the support function of $\mathcal{W}$, namely,
\begin{align*}
F(\xi) = \sup_{X \in \mathcal{W}} \langle \xi, X \rangle,\quad  \xi\in \mathbb{S}^{n-1}.
\end{align*}

Now we introduce the anisotropic area of a smooth oriented hypersurface $X:N\to M\subset \R^n$.
\begin{defn}
	Let $M \subset \mathbb{R}^n$ be a smooth oriented hypersurface and $\nu$ be its unit normal vector. We define the \emph{anisotropic area} of $M$ as $|M|_F := \int_{M} F(\nu) \,d\mu$. Denote by $d\mu_F=F(\nu)d\mu$ the \emph{anisotropic area element} of $M$.
\end{defn}

\begin{rem}
For $M=\pt \mathcal{W}$, we can check by the divergence theorem that
\begin{equation*}
|\partial \mathcal{W}|_F := \int_{\partial \mathcal{W}} F(\nu) \,d\mu = \int_{\partial \mathcal{W}} \langle X, \nu\rangle \,d\mu = \int_{\mathcal{W}} \mathrm{div} X\,dx = n|\mathcal{W}|.
\end{equation*}
\end{rem}

Next we introduce the anisotropic Gauss map for an oriented hypersurface in $\R^n$ (for example, see \cite{he-li2008}).
\begin{defn}
The \emph{anisotropic Gauss map} $\nu_F: M \to \pt \mathcal{W}$ from an oriented hypersurface $M$ in $\R^{n}$ to the Wulff shape $\pt\mathcal{W}$ is defined by
\begin{align*}
\nu_F : \ & M \to \pt \mathcal{W},\\
&X \mapsto DF(\nu(X))= F(\nu(X)) \nu(X) + \nabla^{\mathbb{S}^{n-1}} F (\nu(X)),
\end{align*}
where $\nu$ is the unit normal vector of $M$.
\end{defn}
\begin{rem}
The vector $\nu_F$ is also called the anisotropic unit normal of the hypersurface.
\end{rem}

Let $A_F$ be the 2-tensor on $\SS^{n-1}$ defined by
\begin{align*}
A_F(\xi) =( \nabla^{\mathbb{S}^{n-1}})^2 F(\xi) + F(\xi) \sigma,\quad \xi \in \SS^{n-1},
\end{align*}
where $\sigma$ is the standard metric on $\mathbb{S}^{n-1}$.

\begin{defn}
The \emph{anisotropic principal curvatures} $\kappa_1^F, \dots, \kappa_{n-1}^F$ of a smooth oriented hypersurface $M$ in $\R^n$ are defined as the eigenvalues of the tangent map
\begin{equation*}
d\nu_F:T_XM\to T_{\nu_F(X)}\PW\cong T_XM.
\end{equation*}
The \emph{anisotropic mean curvature} is defined as
\begin{align*}
H_F := \mathrm{tr}(d\nu_F) =\sum_i \kappa_i^F = \sum_{i, j, k} (A_F)_{i}^{j} \left(\nu(X)\right) g^{ik}(X) h_{kj}(X),
\end{align*}
where $g$ and $h$ are the first and second fundamental forms of the hypersurface respectively. We call $M$ anisotropically mean convex or $F$-mean convex if $H_F>0$ on $M$.
\end{defn}
Now we are ready to state the result concerning the inverse anisotropic mean curvature flow in the Euclidean space. In \cite{Xia17}, following the (isotropic) works \cite{Ger90,Urb90}, Chao Xia considered the inverse anisotropic mean curvature flow for a star-shaped $F$-mean convex hypersurface, i.e., the parabolic evolution equation
\begin{align}\label{eq-IAMCF}
\partial_t X = \frac{1}{H_F} \nu_F,
\end{align}
where $X:N\times [0,T)\to \R^n$ is a smooth family of embeddings from a closed manifold $N$ to $\R^n$. He proved the following result on this flow.
\begin{thm}[\cite{Xia17}]
Let $M$ be a smooth compact star-shaped and $F$-mean convex hypersurface without boundary in $\R^n$ ($n\geq  3$). Then the inverse anisotropic mean curvature flow starting from $M$ exists for all time, and converges smoothly and exponentially to an expanded Wulff shape determined by the initial hypersurface $M$.
\end{thm}

\section{Upper bounds in $\mathbb{H}^n$ via the inverse mean curvature flow}\label{sec3}

Consider the inverse mean curvature flow
\begin{equation}
\pt_t X=\frac{1}{\sigma_1}\nu
\end{equation}
in the hyperbolic space $\H^n$ with the initial hypersurface $\pt K$ and the resulting hypersurfaces $M_t$, where $\sigma_1$ is the mean curvature of $M_t$. Then we get (see Section~3 in \cite{LX22} for a similar derivation)
\begin{align}
\mathrm{Cap}_p(K)\leq \left(\int_0^\infty T_p^{-1/(p-1)}(t)dt\right)^{1-p},
\end{align}
and
\begin{equation*}
T_p(t)=\int_{M_t}\sigma_1^{p-1}d\mu_t.
\end{equation*}
In the following we consider three cases separately.
\subsection{The case $n\geq 2$ and $p=2$} For this case we need to estimate $\int_{M_t}\sigma_1d\mu_t$.

Recall the quermassintegrals
\begin{align*}
W_0(K)&:=|K|,\\
W_1(K)&:=\frac{1}{n}|\pt K|,\\
W_2(K)&:=\frac{1}{n(n-1)}\int_{\pt K}\sigma_1d\mu-\frac{1}{n}|K|.
\end{align*}
Under the inverse mean curvature flow, using the variational formula (3.5) in \cite{WX14} we get (for $n=2$, the right-hand side of the first line below is understood to be $0$; see the simplified argument at the end of this part)
\begin{align*}
\frac{d}{dt}W_2(K_t)&=\frac{n-2}{n}\int_{\pt K_t}\frac{\sigma_2}{C_{n-1}^2}\frac{1}{\sigma_1}d\mu_t\\
&\leq  \frac{n-2}{n(n-1)^2}\int_{\pt K_t}\sigma_1d\mu_t\\
&=\frac{n-2}{n-1}\left(\frac{1}{n(n-1)}\int_{\pt K_t}\sigma_1d\mu_t-\frac{1}{n}|K_t|\right)+\frac{n-2}{n(n-1)}|K_t|\\
&=\frac{n-2}{n-1}W_2(K_t)+\frac{n-2}{n(n-1)}|K_t|.
\end{align*}
Here $K_t$ is the compact set enclosed by $M_t$ so that $\pt K_t=M_t$. Let $I$ be the isoperimetric function on $\H^n$, i.e., $|B_r|=I(|\pt B_r|)$. Then
\begin{equation}
|K_t|\leq I(|\pt K_t|)=I(e^t |\pt K_0|),
\end{equation}
where we used $|\pt K_t|=e^t|\pt K_0|$.

So we obtain
\begin{align*}
W_2(K_t)\leq e^{\frac{n-2}{n-1}t}\left(W_2(K_0)+\frac{n-2}{n(n-1)}\int_0^t e^{-\frac{n-2}{n-1}\tau}I(e^\tau |\pt K_0|)d\tau\right).
\end{align*}

Then we get an upper bound for $\int_{\pt K_t}\sigma_1d\mu_t$ as
\begin{align*}
&\frac{1}{n(n-1)}\int_{\pt K_t}\sigma_1d\mu_t=W_2(K_t)+\frac{1}{n}|K_t|\\
&\leq e^{\frac{n-2}{n-1}t}\left(W_2(K_0)+\frac{n-2}{n(n-1)}\int_0^t e^{-\frac{n-2}{n-1}\tau}I(e^\tau |\pt K_0|)d\tau\right)+\frac{1}{n}I(e^t |\pt K_0|),
\end{align*}
where again we used $|\pt K_t|=e^t|\pt K_0|$.

Then we get
\begin{align*}
&\qquad \qquad \mathrm{Cap}_p(K)\leq n(n-1)\times  \\
& \left(\int_0^\infty \left(e^{\frac{n-2}{n-1}t}\left(W_2(K_0)+\frac{n-2}{n(n-1)}\int_0^t e^{-\frac{n-2}{n-1}\tau}I(e^\tau |\pt K_0|)d\tau\right)+\frac{1}{n}I(e^t |\pt K_0|)\right)^{-1}dt\right)^{-1}.
\end{align*}

In particular, when $n=2$, $W_2(K_t)$ is a constant $\pi$ by the Gauss--Bonnet theorem. So
\begin{align*}
\int_{\pt K_t}\sigma_1d\mu_t&=2\pi +|K_t|\leq \sqrt{4\pi^2+e^{2t}|\pt K_0|^2},
\end{align*}
 since by $|\pt B_r|^2=4\pi |B_r|+|B_r|^2$ we know
\begin{equation}
I(s)=\sqrt{4\pi^2+s^2}-2\pi.
\end{equation}
Then we get
\begin{align*}
\mathrm{Cap}_2(K)&\leq \left(\int_0^\infty (4\pi^2+e^{2t}|\pt K_0|^2)^{-1/2}dt\right)^{-1}=\frac{2\pi}{\mathrm{arsinh}(2\pi/|\pt K_0|)}.
\end{align*}

Last, assume the equality holds. Then the equality case of the isoperimetric inequality used in the above argument implies that the boundary $\pt K_0$ must be a geodesic sphere. So we finish the proof of Theorem~\ref{thm1}.

\subsection{The case $n=2$ and $p\geq 3$}

In this case we get
\begin{align*}
&\frac{d}{dt}T_p(t)=\frac{d}{dt}\int_{M_t}\sigma_1^{p-1}d\mu_t\\
&=\int_{M_t}(p-1)\sigma_1^{p-2}\left(-\Delta \frac{1}{\sigma_1}-\sigma_1+\frac{1}{\sigma_1}\right)+\sigma_1^{p-1}d\mu_t\\
&=-(p-1)(p-2)\int_{M_t}\sigma_1^{p-5}|\nabla \sigma_1|^2d\mu_t\\
&\quad -(p-2)\int_{M_t}\sigma_1^{p-1}d\mu_t+(p-1)\int_{M_t} \sigma_1^{p-3}d\mu_t.
\end{align*}
Noting $p\geq 3$, we see
\begin{align*}
&\frac{d}{dt}T_p(t)\leq -(p-2)\int_{M_t}\sigma_1^{p-1}d\mu_t+(p-1)\int_{M_t} \sigma_1^{p-3}d\mu_t.
\end{align*}
When $p=3$, noting $|M_t|=e^t|M|$, we can get
\begin{align*}
\int_{M_t}\sigma_1^2d\mu_t\leq e^{-t}\int_M \sigma_1^2d\mu_t+2\sinh t |M|,
\end{align*}
or
\begin{align*}
|M_t|\int_{M_t}\sigma_1^2d\mu_t-|M_t|^2\leq |M|\int_M\sigma_1^2d\mu-|M|^2.
\end{align*}
%
Next consider $p>3$. Then by the H\"{o}lder inequality we obtain
\begin{align*}
\int_{M_t} \sigma_1^{p-3}d\mu_t&\leq \left(\int_{M_t}\sigma_1^{(p-3)\frac{p-1}{p-3}}d\mu_t\right)^{\frac{p-3}{p-1}}|M_t|^{\frac{2}{p-1}}=\left(\int_{M_t}\sigma_1^{p-1}d\mu_t\right)^{\frac{p-3}{p-1}}|M_t|^{\frac{2}{p-1}}.
\end{align*}
So we get
\begin{align*}
&\frac{d}{dt}T_p(t)\leq -(p-2)T_p(t)+(p-1)\left(T_p(t)\right)^{\frac{p-3}{p-1}}|M_t|^{\frac{2}{p-1}}.
\end{align*}
Solving it yields
\begin{align*}
\left(|M_t|^{p-2}T_p(t)\right)^{\frac{2}{p-1}}-|M_t|^2\leq \left(|M|^{p-2}T_p(0)\right)^{\frac{2}{p-1}}-|M|^2.
\end{align*}
Note that this result also covers the case $p=3$.

In summary, for $p\geq 3$, we obtain
\begin{align*}
\left(T_p(t)\right)^{\frac{1}{p-1}}\leq |M_t|^{-\frac{p-2}{p-1}} \left(|M_t|^2+ \left(|M|^{p-2}T_p(0)\right)^{\frac{2}{p-1}}-|M|^2\right)^{\frac{1}{2}}.
\end{align*}
So we get
\begin{align*}
\mathrm{Cap}_p(K)\leq \left(\int_0^\infty |M_t|^{\frac{p-2}{p-1}} \left(|M_t|^2+ \left(|M|^{p-2}T_p(0)\right)^{\frac{2}{p-1}}-|M|^2\right)^{-\frac{1}{2}}dt\right)^{1-p}.
\end{align*}

Last, assume the equality holds. Then checking the above argument we see that $\sigma_1$ is constant on the flow curve $M_t$. Thus $M_t$ is a geodesic circle and so is $\pt K=M_0$. We finish the proof of Theorem~\ref{thm2}.

\subsection{The case $n=3$ and $1<p\leq 3$}

Let us first estimate $\int_{M_t}\sigma_1^2d\mu_t$ using the inverse mean curvature flow.

We compute
\begin{align*}
&\frac{d}{dt}\int_{M_t}\sigma_1^2d\mu_t=\int_{M_t}2\sigma_1\left(-\Delta \frac{1}{\sigma_1}-\frac{1}{\sigma_1}|h|^2+\frac{2}{\sigma_1}\right)+\sigma_1^2d\mu_t\\
&\leq \int_{M_t}-2|h|^2+4+\sigma_1^2d\mu_t.
\end{align*}
Now use the Gauss equation $2K_M=\sigma_1^2-|h|^2-2$ to replace $|h|^2$. Here $K_M$ is the Gauss curvature for $M_t$ so that $\int_{M_t}K_Md\mu_t=4\pi$. Therefore
\begin{align*}
&\frac{d}{dt}\int_{M_t}\sigma_1^2d\mu_t\leq \int_{M_t}(4K_M-2\sigma_1^2+4)+4+\sigma_1^2d\mu_t\\
&=-\int_{M_t}\sigma_1^2d\mu_t+8e^t|M|+16\pi,
\end{align*}
which implies
\begin{align*}
\frac{d}{dt}\left(e^{t}\int_{M_t}\sigma_1^2d\mu_t\right)\leq e^{t}\left(8e^t|M|+16\pi\right).
\end{align*}
Here we may rewrite the above inequality as
\begin{align*}
\frac{d}{dt}\left(|M_t|\int_{M_t}\sigma_1^2d\mu_t-4|M_t|^2-16\pi |M_t|\right)\leq 0.
\end{align*}
So we can define the modified Hawking mass in the hyperbolic space by
\begin{align*}
m_H(M_t):&=|M_t|\left(16\pi+4|M_t|-\int_{M_t}\sigma_1^2d\mu_t\right),
\end{align*}
which is a non-decreasing quantity in $t$.
\begin{rem}
The modified Hawking mass in the hyperbolic space was introduced by Hung and Wang in \cite{HW15} and its monotonicity along the inverse mean curvature flow was also proved there. Here we include the proof of the monotonicity for the convenience of readers.
\end{rem}
\begin{rem}
Using the Gauss equation and the Gauss--Bonnet theorem, we see $m_H(M_t)\leq 0$ and $m_H(M_t)=0$ if and only if the surface $M_t$ is a geodesic sphere.
\end{rem}

So we have
\begin{align*}
&\int_{M_t}\sigma_1^2d\mu_t\leq \left(\int_M \sigma_1^2d\mu-4|M|-16\pi\right)e^{-t}+4|M|e^{t}+16\pi.
\end{align*}
Now for $1<p\leq 3$ we can derive
\begin{align*}
T_p(t)&=\int_{M_t}\sigma_1^{p-1}d\mu_t\leq \left(\int_{M_t}\sigma_1^2d\mu_t\right)^{\frac{p-1}{2}}|M_t|^{\frac{3-p}{2}}.
\end{align*}
So
\begin{align*}
T_p(t)&\leq \left(\left(\int_M \sigma_1^2d\mu-4|M|-16\pi\right)e^{-t}+4|M|e^{t}+16\pi\right)^{\frac{p-1}{2}}e^{\frac{3-p}{2}t}|M|^{\frac{3-p}{2}}.
\end{align*}
Then
\begin{align*}
&\mathrm{Cap}_p(K)\leq \left(\int_0^\infty T_p^{-1/(p-1)}(t)dt\right)^{1-p}\\
&\leq \left(\int_0^\infty\left(\left(\int_M \sigma_1^2d\mu-4|M|-16\pi\right)e^{-t}+4|M|e^{t}+16\pi\right)^{-\frac{1}{2}}e^{-\frac{3-p}{2(p-1)}t}dt\right)^{1-p}|M|^{\frac{3-p}{2}}.
\end{align*}

Last, assume the equality holds. Then we can check that the modified Hawking mass $m_H(M_t)$ is constant along the inverse mean curvature flow, which implies that the flow hypersurfaces $M_t$ are geodesic spheres. So $\pt K=M_0$ is a geodesic sphere and we finish the proof of Theorem~\ref{thm3}.

\begin{rem}
In particular, for $p=2$, we get
\begin{align*}
&\mathrm{Cap}_2(K)\leq  \left(\int_0^\infty\left(a e^{-t}+be^{t}+c\right)^{-\frac{1}{2}}e^{-\frac{1}{2}t}dt\right)^{-1}|M|^{\frac{1}{2}},
\end{align*}
where
\begin{align*}
a&=\int_M \sigma_1^2d\mu-4|M|-16\pi, \quad b=4|M|,\quad c=16\pi.
\end{align*}
If $4ab>c^2$, we get
\begin{align*}
&\int_0^\infty\left(a e^{-t}+be^{t}+c\right)^{-\frac{1}{2}}e^{-\frac{1}{2}t}dt=-\frac{1}{\sqrt{a}}\mathrm{arsinh}\frac{2ae^{-t}+c}{\sqrt{4ab-c^2}}\bigg|_0^\infty\\
&=\frac{1}{\sqrt{a}}\mathrm{arsinh}\frac{2a+c}{\sqrt{4ab-c^2}}-\frac{1}{\sqrt{a}}\mathrm{arsinh}\frac{c}{\sqrt{4ab-c^2}}.
\end{align*}
If $4ab<c^2$, we get
\begin{align*}
&\int_0^\infty\left(a e^{-t}+be^{t}+c\right)^{-\frac{1}{2}}e^{-\frac{1}{2}t}dt=-\frac{1}{\sqrt{a}}\mathrm{arcosh}\frac{2ae^{-t}+c}{\sqrt{c^2-4ab}}\bigg|_0^\infty\\
&=\frac{1}{\sqrt{a}}\mathrm{arcosh}\frac{2a+c}{\sqrt{c^2-4ab}}-\frac{1}{\sqrt{a}}\mathrm{arcosh}\frac{c}{\sqrt{c^2-4ab}}.
\end{align*}
If $4ab=c^2$, we get
\begin{align*}
&\int_0^\infty\left(a e^{-t}+be^{t}+c\right)^{-\frac{1}{2}}e^{-\frac{1}{2}t}dt=-\frac{1}{\sqrt{a}}\log(2ae^{-t}+c)\bigg|_0^\infty\\
&=\frac{1}{\sqrt{a}}\log(2a+c)-\frac{1}{\sqrt{a}}\log c.
\end{align*}
\end{rem}

\section{The upper bound in $\mathbb{H}^n$ via the unit-speed normal flow}\label{sec4}

In this section we consider a convex hypersurface $M_0$ and use the unit-speed normal flow
\begin{align}\label{eq-normal-flow}
\pt_t X=\nu,
\end{align}
where $X:N\times [0,T)\to \H^n$ is a smooth family of embeddings from a closed manifold $N$ to $\H^n$.

Here we use the hyperboloid model in the Minkowski space $\R^{n,1}$ for $\H^n$. More precisely, the Minkowski space $\R^{n,1}$ is the linear space $\R^{n+1}$ equipped with the Lorentz metric
\begin{equation*}
ds^2=dx_1^2+dx_2^2+\cdots+dx_n^2-dx_{n+1}^2.
\end{equation*}
Then $\H^n$ is viewed as the set
\begin{equation*}
\H^n=\{x\in \R^{n,1}|x_{n+1}=\sqrt{1+x_1^2+x_2^2+\cdots+x_n^2}\}
\end{equation*}
with the induced metric.

So under the unit-speed normal flow \eqref{eq-normal-flow} we have
\begin{align*}
X(p,t)=\cosh t \cdot X(p,0)+\sinh t\cdot  \nu(p,0).
\end{align*}
Let $\{e_i\}_{i=1}^{n-1}$ be orthonormal principal directions in a neighbourhood of a point $X(p,0)$ on $M_0$. Then the vectors
\begin{equation*}
X(\cdot,t)_*(e_i)=(\cosh t+\sinh t\cdot \kappa_i)e_i
\end{equation*}
form an orthogonal frame in a neighbourhood of the point $X(p,t)$. So
\begin{align*}
d\mu_t=\prod_{i=1}^{n-1}(\cosh t+\sinh t\cdot \kappa_i)d\mu=\sum_{i=0}^{n-1}\cosh^{n-1-i}t\cdot \sinh^i t\cdot \sigma_id\mu.
\end{align*}
Therefore
\begin{align*}
T_p(t)=|M_t|=\int_M \sum_{i=0}^{n-1}\cosh^{n-1-i}t\cdot \sinh^i t\cdot \sigma_id\mu,
\end{align*}
and then
\begin{align*}
&\mathrm{Cap}_p(K)\leq \left(\int_0^\infty T_p^{-1/(p-1)}(t)dt\right)^{1-p}\\
&=\left(\int_0^\infty\left( \int_M \sum_{i=0}^{n-1}\cosh^{n-1-i}t\cdot \sinh^i t\cdot \sigma_id\mu\right)^{-\frac{1}{p-1}} dt\right)^{1-p} .
\end{align*}

Last, direct computation shows that for a geodesic ball $K=\overline{B_r}$, the above inequality becomes an equality. So we finish the proof of Theorem~\ref{thm4}.

\begin{rem}
Consider the special case $n=2$ and $p=2$. Then we can compute to get a more transparent inequality
\begin{align*}
\mathrm{Cap}_2(K)\leq \sqrt{(\int_M\kappa d\mu)^2-|M|^2}\left(\log \frac{\int_M \kappa d\mu+\sqrt{(\int_M\kappa d\mu)^2-|M|^2}}{|M| }\right)^{-1},
\end{align*}
provided $\int_M \kappa d\mu\geq |M|$. For comparison, if $K=\overline{B_r}$, we have the explicit expression
\begin{align*}
\mathrm{Cap}_2(B_r)=2\pi \left(\int_r^\infty \sinh^{-1} tdt\right)^{-1}=4\pi \left(\log \frac{\cosh r+1}{\cosh r-1}\right)^{-1},
\end{align*}
which is equal to the right-hand side of the above inequality.

On the other hand, if $\int_M\kappa d\mu<|M|$, we get
\begin{align*}
\mathrm{Cap}_2(K)\leq \frac{1}{2}\sqrt{|M|^2-(\int_M\kappa d\mu)^2}\left(\arctan \sqrt{ \frac{|M|-\int_M\kappa d\mu}{|M|+\int_M\kappa d\mu}}\right)^{-1}.
\end{align*}
\end{rem}

\section{The upper bound in $\mathbb{R}^3$ via the weak inverse mean curvature flow}\label{sec5}

\begin{proof}[Proof of Theorem~\ref{thm5}]

In the setting of Theorem~\ref{thm5} we choose the test function $f(x)=\bar{f}(\phi(x))$ for some $C^1$ function $\bar{f}:[0,\infty)\to \R$ satisfying $\bar{f}(0)=1$ and $\bar{f}(\infty)=0$ to be determined. It is a classical fact that this kind of functions can be admissible test functions; see e.g. \cite[Definition~1]{BM08} and the remark below \cite[Definition~1.2]{Shu}. Therefore
\begin{align*}
\mathrm{Cap}_p(K)&\leq \int_U |\nabla f|^pdx=\int_U (\bar{f}'(\phi(x)))^p |\nabla \phi|^pdx.
\end{align*}
Using the co-area formula, we get
\begin{align*}
\int_U (\bar{f}'(\phi(x)))^p |\nabla \phi|^pdx&=\int_0^\infty (\bar{f}'(t))^p \int_{\Sigma_t} |\nabla \phi|^{p-1}d\mu_t dt.
\end{align*}
Note that
\begin{align*}
\int_{\Sigma_t} |\nabla \phi|^{p-1}d\mu_t&=\int_{\Sigma_t}H^{p-1}d\mu_t\leq \left(\int_{\Sigma_t}H^2d\mu_t\right)^{(p-1)/2}|\Sigma_t|^{(3-p)/2}\\
&\leq \left(16\pi -e^{-t}\left(16\pi-\int_{\Sigma'}H'^2 d\mu\right)\right)^{(p-1)/2}e^{\frac{3-p}{2}t}|\Sigma'|^{(3-p)/2},
\end{align*}
where we used the H\"{o}lder inequality and the monotonicity of the modified Hawking mass. Here $H'$ denotes the mean curvature of the surface $\Sigma'$. Thus we get
\begin{align*}
\int_U |\nabla f|^pdx&\leq \int_0^\infty (\bar{f}'(t))^p \left(16\pi +e^{-t}\left(\int_{\Sigma'}H'^2 d\mu-16\pi\right)\right)^{(p-1)/2} e^{\frac{3-p}{2}t}dt\\
&\quad \times |\Sigma'|^{(3-p)/2}.
\end{align*}
Meanwhile, note that by the H\"{o}lder inequality, we have
\begin{align*}
1&=(\bar{f}(0))^p=\left(-\int_0^\infty \bar{f}'(t)dt\right)^p\\
&\leq \int_0^\infty(\bar{f}'(t))^p \left(16\pi +e^{-t}\left(\int_{\Sigma'}H'^2 d\mu-16\pi\right)\right)^{(p-1)/2} e^{\frac{3-p}{2}t} dt\\
&\quad \times \left( \int_0^\infty   \left(16\pi +e^{-t}\left(\int_{\Sigma'}H'^2 d\mu-16\pi\right)\right)^{-1/2}e^{-\frac{3-p}{2(p-1)}t} dt\right)^{p-1},
\end{align*}
with the equality when
\begin{equation*}
\bar{f}'(t)=c\left(16\pi +e^{-t}\left(\int_{\Sigma'}H'^2 d\mu-16\pi\right)\right)^{-1/2}e^{-\frac{3-p}{2(p-1)}t},\quad c\in \R.
\end{equation*}
Set
\begin{equation*}
s':=\frac{\int_{\Sigma'}H'^2 d\mu}{16\pi}-1.
\end{equation*}
So noticing $\bar{f}(0)=1$ and $\bar{f}(\infty)=0$ we may choose
\begin{equation*}
\bar{f}(t)=\frac{\int_t^\infty \left(1 +s'e^{-\tau }\right)^{-1/2}e^{-\frac{3-p}{2(p-1)}\tau}d\tau}{\int_0^\infty \left(1 +s'e^{-\tau}\right)^{-1/2}e^{-\frac{3-p}{2(p-1)}\tau}d\tau}.
\end{equation*}
Then in this case we get
\begin{align*}
\int_U |\nabla f|^pdx&\leq \left( \int_0^\infty   \left(16\pi +e^{-t}\left(\int_{\Sigma'}H'^2 d\mu-16\pi\right)\right)^{-1/2}e^{-\frac{3-p}{2(p-1)}t} dt\right)^{1-p}\\
&\quad \times |\Sigma'|^{(3-p)/2}\\
&=(16\pi)^{\frac{p-1}{2}}\left( \int_0^\infty   \left(1 +s'e^{-t}\right)^{-1/2}e^{-\frac{3-p}{2(p-1)}t} dt\right)^{1-p}\cdot |\Sigma'|^{(3-p)/2}.
\end{align*}
As a result, we have
\begin{align*}
\mathrm{Cap}_p(K)&\leq (16\pi)^{\frac{p-1}{2}}\left( \int_0^\infty   \left(1 +s'e^{-t}\right)^{-1/2}e^{-\frac{3-p}{2(p-1)}t} dt\right)^{1-p}\cdot |\Sigma'|^{(3-p)/2}.
\end{align*}

Next we replace $\Sigma'$ by $\Sigma$. First recall that $\Sigma'$ strictly minimizes area among all surfaces homologous to $\Sigma$. So $|\Sigma'|\leq |\Sigma|$.

Second, because $\Sigma$ is $C^2$, the surface $\Sigma'$ is $C^{1,1}$ and moreover $C^\infty$ where $\Sigma'$ does not contact $\Sigma$. Besides, the mean curvature $H'$ of $\Sigma'$ satisfies
\begin{align*}
H'=0\text{ on }\Sigma'\setminus \Sigma, \text{ and } H'=H\geq 0\text{ } a.e. \text{ on }\Sigma'\cap \Sigma.
\end{align*}
Therefore we see
\begin{equation*}
\int_{\Sigma'}H'^2 d\mu\leq \int_\Sigma H^2 d\mu.
\end{equation*}
In conclusion, we derive
\begin{align*}
\mathrm{Cap}_p(K)&\leq (16\pi)^{\frac{p-1}{2}}\left( \int_0^\infty   \left(1 +se^{-t}\right)^{-1/2}e^{-\frac{3-p}{2(p-1)}t} dt\right)^{1-p}\cdot |\Sigma|^{(3-p)/2},
\end{align*}
where
\begin{equation*}
s:=\frac{\int_{\Sigma}H^2 d\mu}{16\pi}-1.
\end{equation*}
If $\int_\Sigma H^2 d\mu=16\pi$, then $\Sigma$ is a round sphere (see, e.g., \cite{Che71,RS10,Wil68}) with radius $r_0>0$ and we can compute to get
\begin{equation*}
\mathrm{Cap}_p(K)=\left(\frac{3-p}{p-1}\right)^{p-1}4\pi r_0^{3-p}.
\end{equation*}
If $\int_\Sigma H^2 d\mu>16\pi$, then $s>0$ and we take the change of variables
\begin{equation*}
se^{-t}=r^{\frac{2(p-1)}{3-p}}.
\end{equation*}
Then we get
\begin{align*}
\mathrm{Cap}_{ p}(K) &\leq  \left( \frac{3-p}{p-1} \right)^{p-1} (4\pi)^{\frac{p-1}{2}} |\Sigma|^{\frac{3-p}{2}} s^{\frac{3-p}{2}}  \theta^{1-p},
\end{align*}
where
\begin{align*}
\theta := \int_{0}^{s^{\frac{3-p}{2(p-1)}}} \left(1 +  r^{\frac{2(p-1)}{3-p}}\right)^{-\frac{1}{2}}\,dr.
\end{align*}

Next we claim that if $\int_\Sigma H^2 d\mu>16\pi$, then only the strict inequality can occur
\begin{align*}
\mathrm{Cap}_{ p}(K) &<  \left( \frac{3-p}{p-1} \right)^{p-1} (4\pi)^{\frac{p-1}{2}} |\Sigma|^{\frac{3-p}{2}} s^{\frac{3-p}{2}}  \theta^{1-p}.
\end{align*}
Otherwise assume the equality holds. Then checking the above proof, we see that
\begin{align*}
|\Sigma'|=|\Sigma|, \quad \int_{\Sigma'}H'^2d\mu=\int_\Sigma H^2 d\mu,\quad \widetilde{m}_H(\Sigma_t)=\widetilde{m}_H(\Sigma'), \; \forall t>0,
\end{align*}
and $f(x)=\bar{f}(\phi(x))$ is a $p$-harmonic function on $U$ with $f|_\Sigma=1$ and $f(\infty)=0$. So $\Sigma$ is outer-minimizing and the modified Hawking mass $\widetilde{m}_H(\Sigma_t)$ is equal to $\widetilde{m}_H(\Sigma)$ for all $t$. Moreover, since $f(x)$ is $p$-harmonic in $U$, any level set of $f(x)$ can not have non-empty interior by the strong maximum principle, and so the surfaces $\Sigma_t$ and $\Sigma$ do not jump to $\Sigma_t'$ and $\Sigma'$ respectively, in the sense of \cite{HI01} (meaning $\Sigma_t=\Sigma_t'$ and $\Sigma=\Sigma'$). Next fix any $t>0$ and consider the exterior domain of $\Sigma_t$ in $U$. Using the fact $f(x)$ is constant on $\Sigma_t$ and $\Sigma_t$ is at least $C^1$, by the Hopf boundary lemma, we see that $\nabla f$ never vanishes on $\Sigma_t$. So $\phi=\bar{f}^{-1}\circ f$ is a smooth function on $U$ with $\nabla \phi\neq 0$. It follows that the surfaces $\{\Sigma_t\}$ evolve smoothly by the inverse mean curvature flow. Then the equality case of $\widetilde{m}_H(\Sigma_t)\geq \widetilde{m}_H(\Sigma)$ implies that $H$ is constant on $\Sigma_t$, and so $\Sigma_t$ is a round sphere. So $\Sigma$ itself is a round sphere, which contradicts $\int_\Sigma H^2 d\mu>16\pi$. Therefore when $\int_\Sigma H^2 d\mu>16\pi$, we have the strict inequality
\begin{align*}
\mathrm{Cap}(K)&< \left( \frac{3-p}{p-1} \right)^{p-1} (4\pi)^{\frac{p-1}{2}} |\Sigma|^{\frac{3-p}{2}} s^{\frac{3-p}{2}}  \theta^{1-p}.
\end{align*}

So the proof of Theorem~\ref{thm5} is complete.

\end{proof}

\section{The upper bound in $\mathbb{R}^3$ via the inverse anisotropic mean curvature flow}\label{sec6}

\begin{proof}[Proof of Theorem~\ref{thm6}]

For the proof we use the inverse anisotropic mean curvature flow. Recall the estimate (3.1) in \cite{LX22},
\begin{equation*}
		\mathrm{Cap}_{F, p}(K)\leq \left(\int_0^\infty  T_p^{\frac{1}{1-p}}(t)dt\right)^{1-p},
	\end{equation*}
where (see Section~3.2 in \cite{LX22})
\begin{align}\label{eq-inequality-surface}
T_p(t) &= \int_{\Sigma_t} H_F^{p-1} \,d\mu_F \leq \left(\int_{\Sigma_t} H_F^2\,d\mu_F\right)^{\frac{p-1}{2}} \left( |\Sigma_t|_F \right)^{\frac{3-p}{2}},
\end{align}
after using the H\"{o}lder inequality.

Next we define the modified anisotropic Hawking mass
\begin{equation*}
m_H^F\left( \Sigma_t\right) := \frac{|\Sigma_t|_F}{4 |\PW|_F} \left(1 - \frac{\int_{\Sigma_t} H_F^2\,d\mu_F}{4|\PW|_F}\right).
\end{equation*}
\begin{lem}
Let $\Sigma$ be a compact star-shaped $F$-mean convex surface without boundary in $\R^3$. Along the inverse anisotropic mean curvature flow \eqref{eq-IAMCF} starting from $\Sigma$, the modified anisotropic Hawking mass $m_H^F(\Sigma_t)$ is non-decreasing in $t$. Moreover, if $(d/dt)m_H^F(\Sigma_t)=0$ at some time $t>0$, then $\Sigma$ is a translated scaled Wulff shape.
\end{lem}
\begin{proof}
First recalling the computation in the formula (3.2) in \cite{LX22}, we get (let $p=3$ there)
	\begin{align*}
					&\frac{d}{d t} \int_{\Sigma_t} H_F^2\,d\mu_F= \int_{\Sigma_t}\left( -2 \frac{|\hat{\nabla} H_F|_{\hat{g}}^2}{H_F^2} -2|\hat{h}|_{\hat{g}}^{2} +H_F^2\right) d\mu_F.
	\end{align*}
Here $\hat{g}$ is the anisotropic first fundamental form, $\hat{h}$ is the anisotropic second fundamental form, and $\hat{\nabla}$ is the Levi-Civita connection corresponding to $\hat{g}$ on the surface; see \cite{LX22} for more details.
Using $H_F^2 =2K_F+ |\hat{h}|_{\hat{g}}^2$ and Lemma~25 in \cite{LX22}, we obtain
\begin{align*}
		\frac{d}{d t} \int_{\Sigma_t} H_F^2\,d\mu_F &= \int_{\Sigma_t} \left(-2 \frac{|\hat{\nabla} H_F|_{\hat{g}}^2}{H_F^2} -H_F^{2} + 4K_F \right)d\mu_F\\
		&\le 4|\PW|_F - \int_{\Sigma_t} H_F^{2} \,d\mu_F.
\end{align*}
Therefore in view of $(d/dt)|\Sigma_t|_F=|\Sigma_t|_F$ we conclude
\begin{align*}
	&\frac{d}{dt} \left(  |\Sigma_t|_F  \left( 4|\PW|_F - \int_{\Sigma_t} H_F^2 \,d\mu_F \right) \right) \\
&\ge |\Sigma_t|_F  \left( 4|\PW|_F - \int_{\Sigma_t} H_F^2 \,d\mu_F \right)- |\Sigma_t|_F  \left( 4|\PW|_F - \int_{\Sigma_t} H_F^2 \,d\mu_F \right)\\
&=0.
\end{align*}
So the modified anisotropic Hawking mass is non-decreasing in $t$.

Now assume $(d/dt)m_H^F(\Sigma_t)=0$ at some time $t>0$. Then checking the above argument we see that $H_F$ is constant on $\Sigma_t$, which implies that $\Sigma_t$ is a translated scaled Wulff shape (\cite{HLMG09}). So the initial surface $\Sigma$ is a translated scaled Wulff shape. The proof is complete.

\end{proof}

Note that $m_H^F(\Sigma_t)$ is monotone non-decreasing in $t$. So
\begin{align*}
m_H^F\left(\Sigma\right) \leq m_H^F\left( \Sigma_t\right) = \frac{|\Sigma_t|_F}{4 |\PW|_F} \left(1 - \frac{\int_{\Sigma_t} H_F^2\,d\mu_F}{4|\PW|_F}\right),
\end{align*}
which means
\begin{align*}
\int_{\Sigma_t} H_F^2 \,d\mu_F \le 4 |\PW|_F\left( 1- \frac{4 |\PW|_F}{|\Sigma_t|_F} m_H^F\left(\Sigma\right)\right).
\end{align*}
Setting
\begin{equation}
s= \frac{\int_{\Sigma} H_F^2\,d\mu_F}{4|\PW|_F} -1,
\end{equation}
and noting $|\Sigma_t|_F = |\Sigma|_F e^t$, we obtain
\begin{align*}
\int_{\Sigma_t} H_F^2 \,d\mu_F \le 4 |\PW|_F\left( 1+se^{-t}\right).
\end{align*}
Consequently, from \eqref{eq-inequality-surface} we get	
\begin{align*}
T_p(t) &\le \left( 4 |\PW|_F\right)^{\frac{p-1}{2}}\left(1+se^{-t} \right)^{\frac{p-1}{2}}  \left( |\Sigma_t|_F \right)^{\frac{3-p}{2}}\\
&=(4|\PW|_F)^{\frac{p-1}{2}}|\Sigma|_F^{\frac{3-p}{2}}\left( 1+se^{-t}\right)^{\frac{p-1}{2}}   e^{\frac{3-p}{2}t},
\end{align*}
where we used $|\Sigma_t|_F = |\Sigma|_F e^t$.
	
So we get
\begin{align*}
&\mathrm{Cap}_{F, p}(K)\leq \left(\int_0^\infty  T_p^{\frac{1}{1-p}}(t)dt\right)^{1-p}\\
&\leq (4|\PW|_F)^{\frac{p-1}{2}}|\Sigma|_F^{\frac{3-p}{2}}\left(\int_0^\infty \left( 1+se^{-t} \right)^{-1/2}e^{\frac{3-p}{2(1-p)}t}dt\right)^{1-p}.
\end{align*}	

If $m_H^F(\Sigma)=0$, then by Proposition~26 in \cite{LX22} the surface $\Sigma$ is a translated scaled Wulff shape $r_0\pt \mathcal{W}+x_0$ ($r_0>0$, $x_0\in \R^3$) and we can compute directly to get
\begin{align*}
\mathrm{Cap}_{F, p}\left(K\right) = \left(\frac{3-p}{p-1}\right)^{p-1} |\PW|_F \; r_0^{3-p}.
\end{align*}
%
	
If $m_H^F(\Sigma)<0$, then using the change of variables
\begin{align*}
se^{-t}= r^{\frac{2(p-1)}{3-p}},
\end{align*}
we get by direct computation
\begin{align*}
\mathrm{Cap}_{F, p}(K) &\leq  \left( \frac{3-p}{p-1} \right)^{p-1} |\PW|_F^{\frac{p-1}{2}} |\Sigma|_F^{\frac{3-p}{2}}s^{\frac{3-p}{2}}  \theta^{1-p},
\end{align*}
where
\begin{align*}
\theta := \int_{0}^{s^{\frac{3-p}{2(p-1)}}} \left(1 +  r^{\frac{2(p-1)}{3-p}}\right)^{-\frac{1}{2}}\,dr.
\end{align*}

	Should the equality hold, then \eqref{eq-inequality-surface} must be an equality. The H\"{o}lder inequality becomes an equality, which implies that $H_F$ is constant. So $\Sigma_t$ and then $\Sigma$ are translated scaled Wulff shapes (\cite{HLMG09}), which is impossible. Thus we cannot have the equality in this case. The proof is now complete.

\end{proof}

\begin{proof}[Proof of Corollary~\ref{corr1}]
Let $p = 2$ in Theorem~\ref{thm6}. If $m_H^F(\Sigma) = 0$, then $\int_{\Sigma} H_F^2\,d\mu_F=4|\PW|_F$ and the conclusion follows immediately.

If $m_H^F(\Sigma) < 0$, then $\int_{\Sigma} H_F^2\,d\mu_F>4|\PW|_F$ and we obtain
\begin{equation*}
\mathrm{Cap}_{F,2}\left( K\right) < \sqrt{|\PW|_F |\Sigma|_F}s^{1/2} \cdot \theta^{-1},
\end{equation*}
where
\begin{equation*}
		\theta = \int_{0}^{ \sqrt{s}} \left(1 +  r^2\right)^{-\frac{1}{2}}\,dr = \mathrm{arsinh} \sqrt{ s}\text{ and }s=  \frac{\int_{\Sigma} H_F^2\,d\mu_F}{4|\PW|_F}-1.
\end{equation*}
Thus we get
\begin{equation*}
\mathrm{Cap}_{F,2}( K) < \sqrt{|\PW|_F |\Sigma|_F} \frac{\sqrt{s}}{\mathrm{arsinh}\sqrt{s}}.
\end{equation*}
So we finish the proof of Corollary~\ref{corr1}.
\end{proof}


\bibliographystyle{Plain}

\end{document}